\documentclass [10pt,reqno]{amsart}
\usepackage{ucs}
\usepackage{amsmath, amssymb, amscd}
\usepackage{pb-diagram}
\usepackage[latin1]{inputenc}
\usepackage{amsmath}
\usepackage{amssymb}
\usepackage{epsfig}
\usepackage{graphicx}
\usepackage{psfrag}

\newtheorem{theorem}{Theorem}[section]
\newtheorem{definition}[theorem]{Definition}
\newtheorem{lemma}[theorem]{Lemma}

\newtheorem{prop}[theorem]{Proposition}
\newtheorem{proposition}[theorem]{Proposition}
\newtheorem{corollary}[theorem]{Corollary}
\newtheorem{conjecture}[theorem]{Conjecture}
\newtheorem{hypothesis}[theorem]{Hypothesis}

\renewcommand{\epsilon}{\varepsilon}

\DeclareMathAlphabet{\mathpzc}{OT1}{pzc}{m}{it}
\usepackage{mathrsfs}

\newcommand{\Z}{\mathbb{Z}}
\newcommand{\C}{\mathbb{C}}

\newcommand{\CP}{\mathbb{P}}

\renewcommand{\qed}{$\hfill \square$ \smallskip \\}
\renewcommand{\phi}{\varphi}

\newcommand{\Hom}{\text{Hom}}

\newcommand{\M}{\mathscr{M}}

\newcommand{\tr}{\text{tr}}

\begin{document}
\thispagestyle{empty}
\title[The quantum sl(N) graph invariant and a moduli space]{The quantum sl(N) graph invariant and a moduli space}
\author{Andrew Lobb \\ Raphael Zentner}

\begin {abstract} 
We associate a moduli problem to a colored trivalent graph; such graphs, when planar, appear in the state-sum description of the quantum $sl(N)$ knot polynomial due to Murakami, Ohtsuki, and Yamada.  We discuss how the resulting moduli space can be thought of a representation variety.  We show that the Euler characteristic of the moduli space is equal to the quantum $sl(N)$ polynomial of the graph evaluated at unity.  Possible extensions of the result are also indicated.
\end {abstract}

\address{Dept of Mathematical Sciences \\ Durham University \\ Science Laboratories \\ South Rd. \\ Durham DH1 3LE \\ UK}
\email{andrew.lobb@durham.ac.uk}
\address {Mathematisches Institut \\ Universit\"at zu K\"oln \\ Weyertal 86-90 \\ D-50931 K\"oln \\ Deutschland}
\email{rzentner@math.uni-koeln.de}

\maketitle

\section*{Introduction}
In their paper \cite{MOY}, Murakami, Ohtsuki, and Yamada gave an interpretation of the quantum $\mathfrak{sl}(N)$ invariant of a knot or link via a state sum model. This is a generalization of the case $N=2$ which is Kauffman's interpretation of the Jones polynomial via the Kauffman bracket.  In MOY's work, however, the states of an oriented knot diagram are now planar oriented colored trivalent graphs.  Based on this model, Khovanov and Rozansky have given link homology theories \cite{KR1} that categorify the $\mathfrak{sl}(N)$ knot invariant.

There are relationships between these quantum invariants and other invariants of knots and links, such as the knot group, representation spaces of knot groups, or the various Floer homology theories associated to knots and links.  Based on observations on $(2,2p+1)$ torus knots, Kronheimer and Mrowka \cite{KM2} have given a relationship between Khovanov homology, which appears as the $N=2$ case of the above mentioned homology theories, and (singular) instanton knot Floer homology

In this paper we relate MOY's polynomial $P_N(\Gamma)$ of colored trivalent graphs $\Gamma$ to a certain moduli space of decorations of the graph which is itself a space of representations of the graph complement in $SU(N)$ with meridional conditions on the edges of the graphs.  Our moduli space can be considered as a subspace of a product of projective spaces, which leads us to think our moduli space may have some relation with the construction of a categorification of the $\mathfrak{sl}(N)$ polynomial by Cautis and Kamnitzer \cite{CK}.
\\

In the first section we state our result and give context for it in terms of a conjectural picture for higher rank instanton knot floer homology, the second section contains the proof of the result.

\subsection*{Acknowledgements}
We thank Hans Boden, Sabin Cautis, and Dmitri Panov for useful general discussions, and Saugata Basu for answering some questions of ours on real varieties.  We thank CIRM and their Research in Pairs program for providing the environment in which this paper was written.

\section{Background and Results}

\subsection{Colored trivalent graphs and MOY moves}
In \cite{MOY}, Murakami, Ohtsuki, and Yamada give an invariant $P_n(\Gamma)$ of colored trivalent planar graphs $\Gamma$, taking values in the ring $\mathbb{Z}[q, q^{-1}]$ with non-negative coefficients, which determines the quantum invariant of a oriented knot colored with the fundamental representation of $\mathfrak{sl}(N)$ as a state sum via the relationship in Figure~\ref{MOYtoknot}.

\begin{figure}
\centerline{
{
\psfrag{thing1}{$= q^{1-n}$}
\psfrag{thing2}{$-q^{-n}$}
\psfrag{thing3}{$=q^{n-1}$}
\psfrag{thing4}{$-q^n$}
\psfrag{2}{$2$}
\psfrag{x4}{$x_4$}
\psfrag{T+(D)}{$T^+(D)$}
\includegraphics[height=2in,width=3.5in]{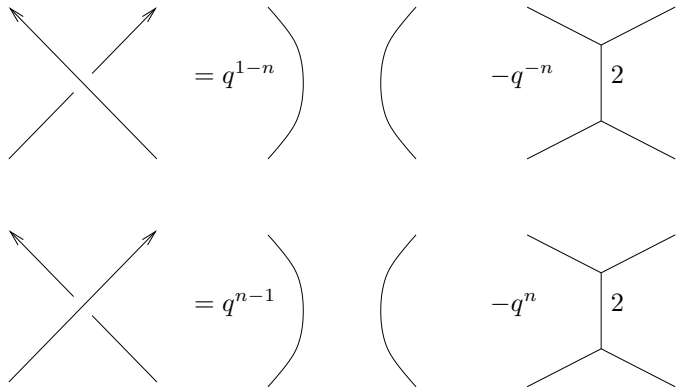}
}}
\caption{We show how the MOY polynomials associated to colored planar trivalent graphs can be used to compute the quantum $\mathfrak{sl}(N)$ polynomials of the knot.  On the left hand side of the local equations are knot diagrams, on the right hand side are trivalent graphs.  We have indicated the edges on the right hand side that are colored with $2$, all other edges on the right hand side are colored with $1$.}
\label{MOYtoknot}
\end{figure}

In Figure \ref{MOYtoknot} the `colors' are drawn from the set $\{ 1,2 \}$ and correspond to associating to each edge of the graph either the fundamental $N$-dimensional representation $V$ of $\mathfrak{sl}(N)$ or the representation $V \wedge V$.  It is important that we restrict the colorings so that meeting at any vertex there are two edges colored with $1$ and one edge colored with $2$.

More generally, in \cite{MOY}, all colors $\wedge^i V$ for $1 \leq i \leq N-1$ are considered, giving state sum interpretations of the quantum $\mathfrak{sl}(N)$ invariants of knots colored with any antisymmetric representation of $\mathfrak{sl}(N)$.  We expect the results of this paper to generalize without too much difficulty to these situations, but here we shall mainly be concerned with the two colors required to give a state sum interpretation of the quantum $\mathfrak{sl}(N)$ polynomial for a knot colored with the standard representation of $\mathfrak{sl}(N)$ as in Figure \ref{MOYtoknot}.

In \cite{MOY} it is important that the graphs $\Gamma$ considered admit a consistent orientation of the edges.  Namely at any trivalent vertex of $\Gamma$ the two 1-colored edges either both point in or both point out, and the 2-colored edge does the opposite.  Another way of saying this is that the total \emph{flux} into a trivalent point (counting the colors as a quantity of flux) should be zero; this is how the consistency condition generalizes to higher antisymmetric powers.  Note that trivalent graphs arising as states of an oriented knot diagram inherit a consistent orientation from the orientation of the knot (see Figure \ref{MOYtoknot}).  From now on whenever we write \emph{colored trivalent graph} we will only mean graphs satisfying such a condition.

The polynomial $P_N(\Gamma)$ associated by Murakami, Ohtsuki, and Yamada to a colored trivalent graph $\Gamma$ can be computed by use of the \emph{MOY moves}.  These are local relationships that look much like the Reidemeister moves.  We give these moves in the next section where we will see that they hold also (up to a shift in some cases or after evaluation at unity in all cases) for the Euler characteristic of a certain moduli space associated to a colored trivalent graph.

\subsection{A moduli problem and the main theorem}

Given a colored trivalent graph $\Gamma$ where each edge is either colored with either $1$ or $2$ as in the last section, decorate each edge colored with $1$ by a point of complex projective $(N-1)$-space $\mathbb{P}^{N-1}$, and each edge colored with $2$ by a point of $\mathbb{G}(2,N)$, the Grassmannian of 2-planes in $\mathbb{C}^N$.  We call this decoration \emph{admissible} if the edge decorations at any trivalent vertex correspond to two orthogonal lines in $\mathbb{C}^N$ and the plane that they span.

\begin{definition}
\label{admissdef}
The set of all admissible decorations of such a colored planar trivalent graph $\Gamma$ forms a moduli space which we denote $\M(\Gamma)$.
\end{definition}

There is a natural generalization of this definition which associates a moduli space to a colored trivalent graph with colors drawn from the set $\{1,2, \ldots, N-1 \}$ such that at any vertex the three edges are colored by $a$, $b$, and $a+b$.  In this situation admissible decorations are decorations of each $a$-colored edge with a point in $\mathbb{G}(a,N)$ such that at at any vertex the three decorations consist of two orthogonal subspaces of $\C^N$ and the subspace that they span.

As we are mainly motivated by the invariants of knots colored with the standard representation of $\mathfrak{sl}(N)$, we shall mostly restrict our attention to the case when we draw colors from the set $\{1,2\}$.  We expand this to the set $\{1,2,3\}$ when doing so is convenient in Subsection \ref{MOYIVsect}.

\begin{figure}
\centerline{
{
\psfrag{a}{$a$}
\psfrag{b}{$b$}
\psfrag{thing3}{$=q^{n-1}$}
\psfrag{thing4}{$-q^n$}
\psfrag{2}{$2$}
\psfrag{x4}{$x_4$}
\psfrag{T+(D)}{$T^+(D)$}
\psfrag{c}{$c$}
\psfrag{d}{$d$}
\psfrag{e}{$e$}
\psfrag{f}{$f$}
\psfrag{g}{$g$}
\psfrag{h}{$h$}
\includegraphics[height=2.5in,width=2in]{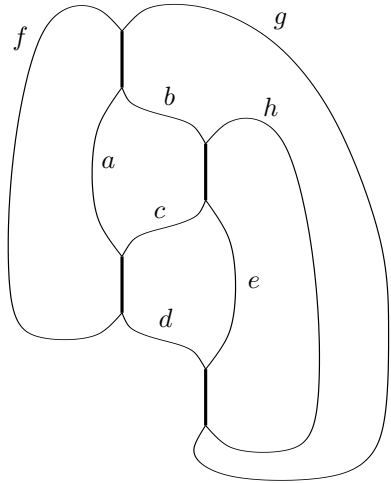}
}}
\caption{We draw a colored graph $G$.  In $G$ the edges colored with $2$ are straight and vertical and have been drawn slightly thicker than the edges colored with $1$.  The labels $a,b,c,d,e,f,g$ refer to points of projective space $\CP^2$ decorating the $1$-colored edges.}
\label{MOYbraidexample}
\end{figure}

In Figure \ref{MOYbraidexample} we draw a colored graph $G$.  As an example, we now determine $H_*(\M(G))$ in the case $N=3$.  In the discussion we sometimes write the ``span'' of points in projective space: we mean the projectivization of the span of the corresponding subspaces of $\C^N$.

\begin{definition}
Given a space $X$ with homology finitely generated and of bounded degree, we write $\pi(X)$ for the Poincare polynomial of the homology of the space:

\[ \pi(X) = \sum_i q^i \dim{H_i(X)} \in \mathbb{Z}[q] {\rm .} \]
\end{definition}

\begin{prop}
\label{countereg}
For the graph $G$ as drawn in Figure \ref{MOYbraidexample} we have

\[ \pi(\M(G)) = (1+3q^2)(1+q^2)(1+q^2+q^4) {\rm .} \]
\end{prop}

\begin{proof}
If we consider $f \in \CP^2$ as fixed while all other decorations are allowed to vary over $\CP^2$ we get a moduli space that we call $\M_1$.  By evaluation of $f$ we see that $\M(G)$ is a fiber bundle over $\CP^2$ with fiber $\M_1$.

With $f$ fixed, we see that $g$ is orthogonal to $f$, and so must lie in a $\CP^1$.  Writing $\M_2$ for the moduli space in which we fix both $f$ and an orthogonal $g$, we see that $\M_1$ is a fiber bundle over $\CP_1$ with fiber $\M_2$.

We focus on $\M_2$.  Observe that $h$ has to lie in the $\CP^1$ consisting of points orthogonal to $g$.  In the case that $h$ is not the unique point of this $\CP^1$ that is also orthogonal to $f$, there is a unique admissible decoration of the rest of the graph, namely $a=f$, $b=c=d=g$, $e=h$.

Suppose now that $h$ is orthogonal to both $f$ and $g$.  The point $e$ has to lie in the $\CP^1$ consisting of points in the span of $g$ and $h$.  Suppose that $e \not= h$.  Then it is easy to check that there is a unique way to decorate the rest of the graph admissibly.  In the case that $e = h$ then any possible choice of $a$ in the $\CP^1$ spanned by $f$ and $g$ determines an admissible decoration of $G$.

This shows that topologically $\M_2 = \vee^3 \CP^1$.  Since this is a space with only even-dimensional homology, as is projective space, we have triviality of all differentials in Serre's spectral sequence computing first $H_*(\M_1)$ and then $H_*(\M(G))$, hence we have the result.
\end{proof}

The motivating hypothesis for this paper is that for a general colored graph $\Gamma$, the Poincare polynomial of the homology $H_*(\M(\Gamma))$ is related to $P_N(\Gamma)$.

\begin{hypothesis}
\label{dreamthm}
We have

\[ q^C \pi(\M(\Gamma)) =  P_N(\Gamma) \rm{,}\]

\noindent for some $C \in \mathbb{Z}$.
\end{hypothesis}

Although this hypothesis looks strong when considering the first few MOY moves, it fails in general.  In fact for the graph $G$ of Proposition \ref{countereg} we have

\[ P_3(G) = (q + q^{-1})^3 (q^{2} + 1 + q^{-2}) \not= q^C \pi(\M(G)) \,\, {\rm for} \,\, {\rm any} \,\, C \in \mathbb{Z} {\rm .} \]

\noindent  Note that this also invalidates the hypothesis for `braid-like' trivalent graphs.  Instead we prove that a version of Hypothesis \ref{dreamthm} holds when restricting our attention to the Euler characteristic.

\begin{theorem}
\label{maintheorem}
For a colored planar trivalent graph $\Gamma$ we have

\[ \chi(\M(\Gamma)) = P_N(\Gamma) (1) {\rm ,} \]

\noindent where $\chi$ denotes Euler characteristic.
\end{theorem}

\noindent We believe in fact that the homology of $\M(\Gamma)$ is supported in even degrees, but we are not yet able to show this.

\begin{conjecture}
\label{evenness}
For any graph $\Gamma$ and for $i$ odd we have $H_i(\M(\Gamma)) = 0$.
\end{conjecture}

\noindent Clearly this conjecture and Theorem \ref{maintheorem} would together imply that Hypothesis \ref{dreamthm} holds when setting $q=1$.

\subsection{Relation to representation spaces}
The complement of a planar graph $\Gamma$ in $S^3$ is homotopy equivalent to a wedge of $k$ circles, where $k$ is equal to the first Betti number of $\Gamma$. Therefore, its fundamental group $G_\Gamma$ is isomorphic to the $k$-fold free product of $\Z$, and the space of homomorphisms $\Hom(G_\Gamma,SU(N))$ is just the $k$-fold product of $SU(N)$. Nonetheless, for a planar trivalent graph this homotopy equivalence does not suggest the most natural generators. Instead we use the presentation given in the following Lemma.

\begin{lemma}
The group $G_\Gamma = \pi_1(S^3 \setminus \Gamma)$ admits a presentation given by 
\[
	\langle \, x_1, \dots, x_m \ | \ R_1, \dots, R_c \ \rangle \ , 
\]
where $m$ is the number of edges, $x_i$ represents a positively oriented meridian to the $i^{th}$ edge, where $c$ is the number of trivalent vertices, and where the relations $R_i$ are the obvious relations at each trivalent vertex. 
\end{lemma}
\qed
We can think, once and for all, that the basepoint is fixed somewhere {\em above} the plane of the graph, e.g. in the eye of the observer. Ignoring the dependence on the basepoint in the sequel will make everything well-defined up to global conjugation. 
\\

We are motivated by the perspective of a potential relationship between the $\mathfrak{sl}(N)$ knot homology theory of Khovanov and Rozansky \cite{KR1} and the instanton Floer homology associated to the group $SU(N)$, as developed by Kronheimer and Mrowka in \cite{KM1}. Therefore, we are only looking at representations for which a certain condition on the conjugacy class of a meridian is satisfied. In \cite[Section 2.5 Example (ii)]{KM1} a coherent condition is that the conjugacy class of a meridian $m$ is sent by a representation $\rho : G_\Gamma \to SU(N)$ to the conjugacy class of the element $\Phi_1$ given by 

\[
\Phi_1:= \, \zeta \, \begin{pmatrix} -1 & 0 & 0 & \dots & 0 \\
								0 & 1 & 0 & \dots & 0 \\
								0 & 0 & 1 & \dots & 0 \\
								\vdots & \vdots & \vdots & \ddots & 0 \\
								0 & 0 & \dots & 0 & 1 
								\end{pmatrix} \ , 
\]
where $\zeta$ is a primitive $N^{th}$ root of $-1$, e.g. $\zeta = \exp(i \pi / N)$. We shall also need another special element of $SU(N)$:
\[
\Phi_2:= \, \zeta^2 \, \begin{pmatrix} -1 & 0 & 0 & \dots & 0 \\
								0 & -1 & 0 & \dots & 0 \\
								0 & 0 & 1 & \dots & 0 \\
								\vdots & \vdots & \vdots & \ddots & 0 \\
								0 & 0 & \dots & 0 & 1 
								\end{pmatrix} \ . 
\]

Notice that for $N=2$ the corresponding meridional condition is satisfied with either orientation on the edge, in contrast to the situation where $N \geq 3$. 
\\

For the purpose of this section, we shall give an orientation of the 2-colored edges. We orient the 2-colored edge coherently so that we can think of the edges coming with a flux of magnitude given by the color of the edge, and so that there are no sources or sinks of flux at any point of the graph. 

We now define a subspace of $\Hom(G_\Gamma,SU(N))$ that we will relate to the moduli space of decorations considered above.
\begin{definition}
Suppose we are given a trivalent oriented graph with 1-colored and 2-colored oriented edges. We denote by $R_{\Phi_1,\Phi_2}(G_\Gamma;SU(N))$ the space of homomorphisms $\rho: G_\Gamma \to SU(N)$ such that for any oriented meridian $m$ to an oriented 1-colored edge $e$, and to any oriented meridian $n$ to a 2-colored edge $E$  we have
\begin{equation*}
	\rho(m) \sim \Phi_1 \, \ \ \text{and} \ \ \, \rho(n) \sim \Phi_2 \ ,
\end{equation*}
i.e. $\rho(m)$ is conjugated to $\Phi_1$ inside $SU(N)$ and so is $\rho(n)$ to $\Phi_2$.
\end{definition}

Suppose now that we are given a trivalent oriented graph $\Gamma$ in the plane with the condition that at any trivalent vertex the two bounding oriented 1-colored edges have either both its endpoint or both its starting point. 
There is a natural map 
\begin{equation}
\label{art deco}
D: R_{\Phi_1,\Phi_2}(G_\Gamma;SU(N))  \to  M(\Gamma)
\end{equation} 
defined in the following way. To a representation $\rho$ we are going to associate the following decoration of $\Gamma$: Let $e$ be an oriented 1-colored edge and $m$ an oriented meridian of $e$. We decorate $e$ by the 1-dimensional eigenspace of $\rho(m) \in SU(N)$ associated to the eigenvalue $-\zeta$. Let $E$ be a 2-colored edge and $n$ an oriented meridian of $E$. We decorate the edge $E$ by the 2-dimensional eigenspace of $\rho(n) \in SU(N)$ associated to the eigenvalue $-\zeta^2$. 

That this map is indeed well-defined is a consequence of the following lemma and the imposed requirement on the orientations of the edges that meet at trivalent vertices.

\begin{lemma}
Let $S, T \in SU(N)$ be two elements that are both conjugate to the element $\Phi_1$ above. Then the composition $ST$ is conjugate to $\Phi_2$ if and only if the $(-\zeta)$ eigenspaces of $S$ and $T$ are orthogonal. 
\end{lemma}
\begin{proof}
We just have to prove the `only if' statement as the other direction is trivial. Observe that we have 
\[
	\tr(\Phi_2) = \zeta^2 (-2 + (N-2)) = \zeta^2 (N-4) \ , 
\]
\noindent and that the trace is clearly an invariant of the conjugacy class of an element in $SU(N)$. 
On the other hand it follows from explicit computation that 
$\tr(ST)$ is equal to this value only if the two 1-dimensional eigenspaces of $S$ and $T$ are orthogonal to each other.

\end{proof}

\begin{proposition}
Suppose we are given a trivalent oriented graph $\Gamma$ in the plane with the condition that at any trivalent point the two bounding oriented 1-colored edges have either both its endpoint or both its starting point. Then the map 
\[ D: R_{\Phi_1,\Phi_2}(G_\Gamma;SU(N))  \to  \M(\Gamma) 
\]
 just defined is a homeomorphism.
\end{proposition}

\begin{proof}
It suffices to give a map $P: \M(\Gamma)   \to R_{\Phi_1,\Phi_2}(G_\Gamma;SU(N)) $ that is inverse to $D$. But the orientations of the edges associates to each element of $\M(\Gamma)$ a representation in $R_{\Phi_1,\Phi_2}(G_\Gamma;SU(N)) $, and the two maps are clearly inverses of each other. 
\end{proof}

\section{Proofs}
In this section we establish Theorem \ref{maintheorem} by considering the MOY moves, so breaking the proof of the theorem into five subsections, one for each move.  The first three of these subsections contain results in terms of fiber bundles and are directed towards Hypothesis \ref{dreamthm} and Conjecture \ref{evenness}.  Unfortunately it is not possible to prove such a result in the last case (see Proposition \ref{countereg} and the following discussion for a counterexample) so we prove results on the Euler characteristic instead.

\subsection{MOY move 0}
The $0^{th}$ MOY move, MOY0, states that the invariant $P_n(U)$ associated to a circle $U$ in the plane colored with $1$ is

\[ P_N(U) = \frac{q^N - q^{-N}}{q - q^{-1}} {\rm .}\]

\begin{prop}
\label{P1}
We have $\M(U) = \mathbb{P}^{N-1}$.
\end{prop}

The proof is immediate.  Note this shows that $U$ satisfies Hypothesis \ref{dreamthm}.

\begin{figure}
\centerline{
{
\psfrag{thing1}{$= q^{1-n}$}
\psfrag{thing2}{$-q^{-n}$}
\psfrag{thing3}{$=q^{n-1}$}
\psfrag{thing4}{$-q^n$}
\psfrag{2}{$2$}
\psfrag{a}{$a$}
\psfrag{b}{$b$}
\psfrag{c}{$c$}
\includegraphics[height=1in,width=1.5in]{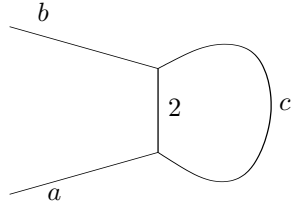}
}}
\caption{We call this subgraph $G$.  The three edges labelled $a,b,c$ are colored with $1$, we have indicated the edge colored with $2$.}
\label{MOY1}
\end{figure}

\subsection{MOY move I}
Suppose that $\Gamma$ is a colored trivalent graph containing the subgraph $G$ defined in Figure \ref{MOY1} .  The $1^{st}$ MOY move, MOY1, states that if $\Gamma'$ is the result of replacing $G$ in $\Gamma$ with a single $1$-colored edge then

\[ P_N(\Gamma) = \frac{q^{N-1} - q^{1-N}}{q - q^{-1}} P_N(\Gamma') {\rm .} \]

\begin{prop}
For $\Gamma$ and $\Gamma'$ as above we have that $\M(\Gamma)$ is a $\CP^{N-2}$-bundle over $\M(\Gamma')$.  
\end{prop}

\begin{proof}
Observe that any admissible choice of decoration of the graph $\Gamma$ when restricted to the subgraph $G$ must have the same decoration $b$ as $a$.  Also, for a choice of decoration $a$, $c$ can be any line in the $(N-1)$-plane orthogonal to $a$.
\end{proof}

\noindent Then immediately we see

\begin{corollary}
\label{P2}
For $\Gamma$ and $\Gamma'$ as above we have that

\[ \chi(\M(\Gamma)) = \chi(\CP^{N-2})\chi(\M(\Gamma')) = \left. \frac{q^{N-1} - q^{1-N}}{q - q^{-1}}\right\vert_{q=1} \chi(\M(\Gamma')) \rm{.} \] \qed
\end{corollary}

Observe further that with the assumption that $\M(\Gamma')$ has only even-dimensional homology we have

\[ \pi(\M(\Gamma)) = \pi(\CP^{n-2}) \pi(\M(\Gamma')) \]

\noindent by the triviality of the differentials in the Serre spectral sequence.  Note that this is in the direction of Hypothesis \ref{dreamthm}.

\subsection{MOY move II}
Suppose that $\Gamma$ is a colored trivalent graph containing as a subgraph $G$ as defined in Figure \ref{MOY2}.  Then the move MOY2 states that if $\Gamma'$ is the result of replacing $G$ in $\Gamma$ by a single edge colored by $2$ then

\[ P_N(\Gamma) = (q + q^{-1}) P_N(\Gamma') {\rm .} \]

\begin{figure}
\centerline{
{
\psfrag{thing1}{$= q^{1-n}$}
\psfrag{thing2}{$-q^{-n}$}
\psfrag{thing3}{$=q^{n-1}$}
\psfrag{thing4}{$-q^n$}
\psfrag{2}{$2$}
\psfrag{a}{$a$}
\psfrag{b}{$b$}
\psfrag{c}{$c$}
\includegraphics[height=2in,width=1in]{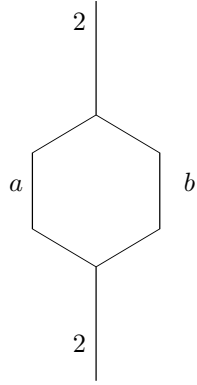}
}}
\caption{We call this subgraph $G$.  The edges labelled $a,b$ are colored with $1$, the other edges are colored with $2$ as indicated.}
\label{MOY2}
\end{figure}

\begin{prop}
For $\Gamma$ and $\Gamma'$ as above we have that $\M(\Gamma)$ is a $\CP^1$-bundle over $\M(\Gamma')$.
\end{prop}

\begin{proof}
First observe that any admissible decoration of $\Gamma$ must decorate each of the the 2-colored edges of $G$ with the same point of $\mathbb{G}(2,N)$.  For such a choice, the decoration $a$ can be any line in the 2-plane decorating the 2-colored edges, and then $b$ must be the unique othogonal line to $a$ in this plane.
\end{proof}

\noindent Then immediately we have

\begin{corollary}
\label{P3}
For $\Gamma$ and $\Gamma'$ as above

\[ \chi(\M(\Gamma)) = \chi(\CP^1)\chi(\M(\Gamma')) = (q + q^{-1})\vert_{q=1} \chi(\M(\Gamma'))  {\rm .} \] \qed
\end{corollary}

Observe further that with the assumption that $\M(\Gamma')$ has only even dimensional homology we have that 

\[ \pi(\M(\Gamma)) = \pi(\CP^1)\pi(\M(\Gamma')) \]

\noindent by the Serre spectral sequence.  Again, this is in direction of Hypothesis \ref{dreamthm}.

\subsection{MOY move III}
Suppose that $\Gamma$ is a colored trivalent graph containing as a subgraph $G$ as defined in Figure \ref{MOY2opp}.  Then the move MOY3 states that if $\Gamma_i$ is the result of replacing $G$ in $\Gamma$ by the subgraphs $G_i$ shown in Figure \ref{MOY2opp} for $i = 1,2$ then

\[ P_N(\Gamma) = \frac{q^{N-2} - q^{2-N}}{q + q^{-1}} P_N(\Gamma_1) + P_N(\Gamma_2) {\rm .} \]

\noindent Note that

\[  \left. \frac{q^{N-2} - q^{2-N}}{q + q^{-1}} \right|_{q=1} = N-2 {\rm .} \]

\begin{figure}
\centerline{
{
\psfrag{thing1}{$= q^{1-n}$}
\psfrag{thing2}{$-q^{-n}$}
\psfrag{thing3}{$=q^{n-1}$}
\psfrag{thing4}{$-q^n$}
\psfrag{a}{$a$}
\psfrag{b}{$b$}
\psfrag{c}{$c$}
\psfrag{d}{$d$}
\psfrag{e}{$e$}
\psfrag{f}{$f$}
\psfrag{2}{$2$}
\psfrag{G}{$G$}
\psfrag{G1}{$G_1$}
\psfrag{G2}{$G_2$}
\includegraphics[height=1.4in,width=4.4in]{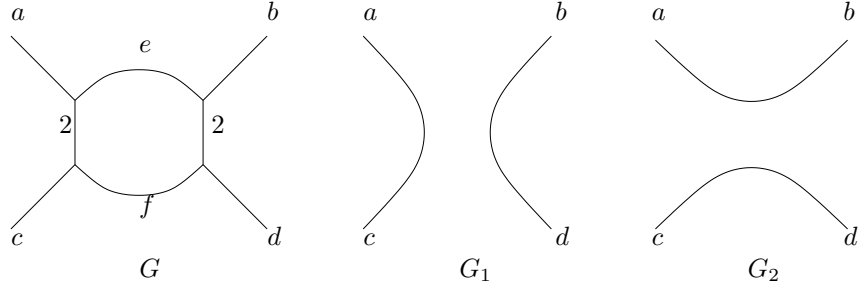}
}}
\caption{We show three subgraphs $G$, $G_1$, $G_2$ of colored trivalent graphs.  Each edge decorated with a lower case English letter is 1-colored, those shown with $2$ are 2-colored.}
\label{MOY2opp}
\end{figure}

\begin{prop}
\label{hisonalwhatsup}
For $\Gamma$, $\Gamma_1$, $\Gamma_2$ as described above we have

\[ \chi(\M(\Gamma)) =  (N-2) \chi(\M(\Gamma_1)) + \chi(\M(\Gamma_2)) {\rm .}\]
\end{prop}

\begin{proof}
Consider the subgraph $G$ of $\Gamma$ as shown in Figure \ref{MOY2opp}.  It is easy to see that for an admissible decoration of $\Gamma$ we must have either ($a=b$ and $c=d$) or ($a=c$ and $b=d$).  In the case when $a=c\not=b=d$, there is a $\CP^{N-3}$ of choices of decoration for the interior edges of $G$ (this $\CP^{N-3}$ corresponds to the projectivization of the $(N-2)$-plane perpendicular to the $2$-plane spanned by the lines $a$ and $b$).  In the case when $a=b\not=c=d$, there is a unique choice of decoration for the interior edges of $G$.  And finally, in the case when $a=b=c=d$, there is a $\CP^{N-2}$ of choices for the interior edges of $G$ (this $\CP^{N-2}$ being the projectivization of the othogonal complement to the line $a$).

Clearly for admissible decorations of $\Gamma_1$ we must have $a=c$ and $b=d$, and for admissible decorations of $\Gamma_2$ we must have $a=b$ and $c=d$.

Note that our moduli spaces $\M(\Gamma)$, $\M(\Gamma_1)$, $\M(\Gamma_2)$ naturally have the structure of compact real varieties embedded in a product of complex projective spaces (the orthogonality condition at a trivalent vertex is a real polynomial condition not a complex condition).

By evaluation at $a,b,c,d$ we get algebraic maps to $(\CP^{N-1})^4$ from each of $\M(\Gamma)$, $\M(\Gamma_1)$, $\M(\Gamma_2)$.  We write $V$, $V_1$, $V_2$ for the preimages of the subvariety $\Delta$ of $(\CP^{N-1})^4$ given by $a=b=c=d$.

Since we are dealing with an algebraic map of real varieties, for a small enough open set $U \subset (\CP^{N-1})^4$ containing $\Delta$ we have that the respective preimages $\widetilde{V}$, $\widetilde{V_1}$, $\widetilde{V_2}$ of $U$ are each homotopy equivalent to $V$, $V_1$, $V_2$ respectively.  This is a standard argument using Hardt triviality.

Note also that $V_1 = V_2$ naturally (although we will retain both indices for notational convenience) and $V$ is a $\CP^{N-2}$-bundle over $V_1 = V_2$ so that

\[ \chi(\widetilde{V}) =  \chi(V) = (N-1) \chi(V_1) = (N-2) \chi(V_1) + \chi(V_2) = (N-2) \chi(\widetilde{V_1}) + \chi(\widetilde{V_2}) {\rm .} \]

\noindent Using the Mayer-Vietoris sequence we see that

\begin{eqnarray*}
\chi(\M(\Gamma)) &=& \chi(\M(\Gamma) \setminus V) + \chi(\tilde{V}) - \chi((\M(\Gamma) \setminus V)\cap \widetilde{V}) \\
&=& \chi(\M(\Gamma) \setminus V) + (N-2) \chi(\widetilde{V_1}) + \chi(\widetilde{V_2})  -  \chi(\widetilde{V} \setminus V) {\rm .}
\end{eqnarray*}

Now note that $\M(\Gamma) \setminus V$ is the disjoint union of $\M(\Gamma_2) \setminus V_2$ and a $\CP^{N-3}$-bundle over $\M(\Gamma_1) \setminus V_1$, and $\widetilde{V} \setminus V$ is the disjoint union of $\widetilde{V_2} \setminus V_2$ and a $\CP^{N-3}$-bundle over $\widetilde{V_1} \setminus V_1$.  Hence we have

\begin{eqnarray*}
\chi(\M(\Gamma)) &=& \chi(\M(\Gamma_2) \setminus V_2) + \chi(\widetilde{V_2}) - \chi(\widetilde{V_2} \setminus V_2) \\
&+& (N-2)\chi(\M(\Gamma_1) \setminus V_1)  + (N-2) \chi(\widetilde{V_1}) - (N-2)\chi(\widetilde{V_1} \setminus V_1) \\
&=& (N-2)\chi(\M(\Gamma_1)) + \chi(\M(\Gamma_2)) {\rm .}
\end{eqnarray*}
\end{proof}

\subsection{MOY move IV}
\label{MOYIVsect}

\begin{figure}
\centerline{
{
\psfrag{thing1}{$= q^{1-n}$}
\psfrag{thing2}{$-q^{-n}$}
\psfrag{thing3}{$=q^{n-1}$}
\psfrag{thing4}{$-q^n$}
\psfrag{2}{$2$}
\psfrag{a}{$a$}
\psfrag{b}{$b$}
\psfrag{c}{$c$}
\psfrag{d}{$d$}
\psfrag{e}{$e$}
\psfrag{P}{$P$}
\psfrag{Q}{$Q$}
\psfrag{R}{$R$}
\psfrag{Phi}{$\Phi$}
\psfrag{Gamma}{$G$}
\psfrag{Gamma1}{$G_1$}
\psfrag{Gamma2}{$G_2$}
\includegraphics[height=2in,width=3.5in]{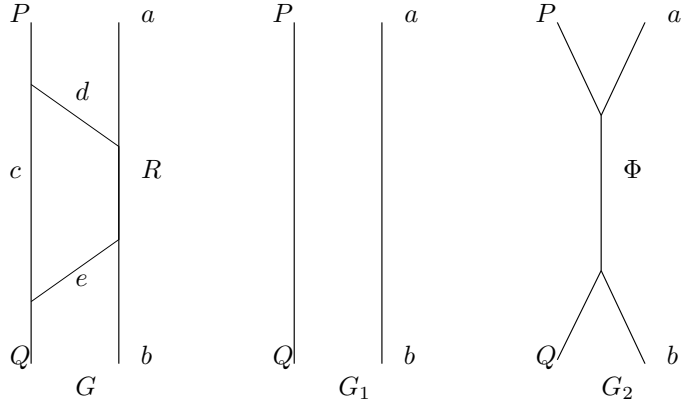}
}}
\caption{We show three subgraphs $G$, $G_1$, $G_2$ of colored trivalent graphs.  Each edge decorated with a lower case English letter is 1-colored, those decorated with upper case English letters are 2-colored, and that decorated with $\Phi$ is 3-colored.}
\label{MOY3}
\end{figure}

\begin{figure}
\centerline{
{
\psfrag{thing1}{$= q^{1-n}$}
\psfrag{thing2}{$-q^{-n}$}
\psfrag{thing3}{$=q^{n-1}$}
\psfrag{thing4}{$-q^n$}
\psfrag{2}{$2$}
\psfrag{a}{$1$}
\psfrag{b}{$1$}
\psfrag{c}{$1$}
\psfrag{d}{$d$}
\psfrag{e}{$e$}
\psfrag{P}{$2$}
\psfrag{Q}{$2$}
\psfrag{R}{$R$}
\psfrag{Phi}{$3$}
\psfrag{Gamma}{$\Gamma$}
\psfrag{Gamma1}{$\Gamma_1$}
\psfrag{Gamma2}{$\Gamma_2$}
\includegraphics[height=2in,width=2in]{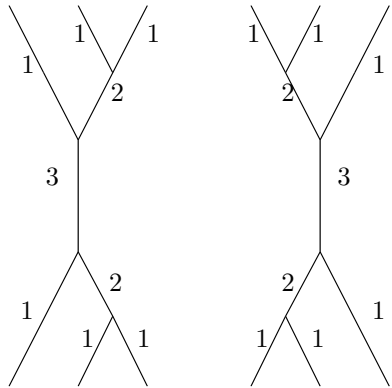}
}}
\caption{If $\Gamma$ is some colored trivalent graph with the left-hand picture appearing as a subgraph and $\Gamma'$ is the result of replacing the subgraph with the right-hand picture, then clearly $\M(\Gamma) = \M(\Gamma')$.}
\label{MOYtrick}
\end{figure}

We refer to the discussion after Definition \ref{admissdef} for the definition of the moduli space associated to a trivalent graph colored with colors from the set $\{1,2,3\}$.

Suppose that $\Gamma$ is a colored trivalent graph with a subgraph $G$ as indicated in Figure \ref{MOY3}.  Form $\Gamma_1$ and $\Gamma_2$ from $\Gamma$ by replacing $G$ with $G_1$ or $G_2$ respectively.  Then the move MOY4 states that

\[ P_N(\Gamma) = P_N(\Gamma_1) + P_N(\Gamma_2) {\rm .} \]

\begin{proposition}
\label{corblimey}
For such graphs $\Gamma$, $\Gamma_1$, $\Gamma_2$ as above we have

\[ \chi(\M(\Gamma))= \chi(\M(\Gamma_1)) + \chi(\M(\Gamma_2)) {\rm .} \]
\end{proposition}

At first sight, this proposition does not seem to give a relationship between trivalent graphs colored only drawing from the palette $\{ 1 , 2 \}$.  However, using the relationship between moduli spaces shown in Figure \ref{MOYtrick} we get such a result.  Murakami, Ohtsuki, and Yamada use an analogous trick in \cite{MOY} to get a relationship between the polynomials $P_N$ of $\{ 1 , 2 \}$-colored graphs.

We note that this proposition is less strong than would be demanded by a proof of Hypothesis \ref{dreamthm}, but the example of Proposition \ref{countereg} shows that Proposition \ref{corblimey} does not admit a lift to the Poincare polynomial in general.

\begin{proof}

By evaluation at the endpoints $a,b,P,Q$ we get algebraic maps from each of $\M(\Gamma)$, $\M(\Gamma_1)$, and $\M(\Gamma_2)$ to $(\CP^{N-1})^2 \times \mathbb{G}(2,N)^2$.  Let $V$, $V_1$, $V_2$ be the respective preimages of the subvariety $\Delta$ of $(\CP^{N-1})^2 \times \mathbb{G}(2,N)^2$ carved out by the three requirements $P=Q$, $a=b$, and $a$ is perpendicular to $P$.

Just as in the proof of \ref{hisonalwhatsup}, for a small enough open set $U \subset (\CP^{N-1})^2 \times \mathbb{G}(2,N)^2$ containing $\Delta$ we have that the respective preimages $\widetilde{V}$, $\widetilde{V_1}$, $\widetilde{V_2}$ of $U$ are each homotopy equivalent to $V$, $V_1$, $V_2$ respectively.

Note that naturally $V_1 = V_2$ but we will retain both indices for notational convenience.

Consider the subgraph $G$ of $\Gamma$.  For $a$ not perpendicular to $P$, one can check that there is a unique way in which to choose decorations for the other edges of $G$.  It follows also in this case that $P=Q$ and $b = a$.

For $a$ perpendicular to $P$, it follows that $Q$ is a plane in the 3-space spanned by $a$ and $P$ and $b$ is the unique perpendicular line to $Q$ in this 3-space.  In this case, when $Q \not= P$ there is a unique choice of decorations for the other edges of $G$, and when $Q = P$ there is a a $\CP^1$ of choices of decorations for the other edges.

Note that this implies that $V$ has the structure of a $\CP^1$-bundle over $V_1 = V_2$ so that

\[ \chi(\widetilde{V}) = \chi(V) = 2\chi(V_1) = \chi(V_1) + \chi(V_2) = \chi(\widetilde{V_1}) + \chi(\widetilde{V_2}) {\rm .} \]

\noindent Hence, using the Mayer-Vietoris sequence, we have

\begin{eqnarray*}
\chi(\M(\Gamma)) &=& \chi(\M(\Gamma) \setminus V) + \chi(\widetilde{V}) - \chi((\M(\Gamma) \setminus V)\cap \widetilde{V}) \\
&=& \chi(\M(\Gamma) \setminus V) + \chi(\widetilde{V_1}) + \chi(\widetilde{V_2}) -  \chi(\widetilde{V} \setminus V) {\rm .}
\end{eqnarray*}

Finally note that $\M(\Gamma) \setminus V$ is the disjoint union of $\M(\Gamma_1) \setminus V_1$ and $\M(\Gamma_2) \setminus V_2$, and also $\widetilde{V} \setminus V$ is the disjoint union of $\widetilde{V_1} \setminus V_1$ and $\widetilde{V_2} \setminus V_2$.  Thus we have

\begin{eqnarray*}
\chi(\M(\Gamma)) &=& \chi(\M(\Gamma_1) \setminus V_1)  + \chi(\widetilde{V_1})   -  \chi(\widetilde{V_1} \setminus V_1) \\
&+& \chi(\M(\Gamma_2) \setminus V_2) + \chi(\widetilde{V_2})  -  \chi(\widetilde{V_2} \setminus V_2) \\
&=& \chi(\M(\Gamma_1)) + \chi(\M(\Gamma_2)) {\rm .}
\end{eqnarray*}
\end{proof}

\subsection{Proof of main theorem}

\begin{proof}[Proof of Theorem \ref{maintheorem}]
Since for any colored trivalent graph $\Gamma$ the value of the polynomial $P_n(\Gamma) \in \mathbb{Z}[q, q^{-1}]$ is determined by the fact that $P_n$ satisfies the five MOY relations it is enough to verify that $\chi(\M (\Gamma))$ satisfies the MOY relations after evaluating at $q=1$.  But we have verified this in Propositions \ref{P1}, \ref{hisonalwhatsup}, \ref{corblimey} and Corollaries \ref{P2}, \ref{P3}.
\end{proof}

\end{document}